\documentclass[10pt]{amsart}
\usepackage{amsmath}
\usepackage{amssymb}
\usepackage{amsfonts}
\usepackage{amsthm}
\usepackage{amsxtra}
\usepackage{rotating}
\usepackage{nicefrac}
\usepackage{float}
\usepackage[all,web,arc,poly,dvips]{xy}
\usepackage{epic,eepic}
\usepackage{epsfig}
\usepackage[dvips]{color}
\usepackage{array}
\usepackage{pifont}
\usepackage{multirow}

\theoremstyle{plain}
\newtheorem{theorem}{Theorem}[section]
\newtheorem{corollary}{Corollary}[section]
\newtheorem{lemma}{Lemma}[section]

\theoremstyle{definition}

\theoremstyle{remark}

\newcommand{\C}{\mathbb C}
\newcommand{\Z}{\mathbb Z}

\newcommand{\VI}{{\rm\scriptscriptstyle VI}}
\newcommand{\V}{{\rm\scriptscriptstyle V}}
\newcommand{\PV}{${\rm P}_{\rm V}\;$}
\newcommand{\PVI}{${\rm P}_{\rm VI}\;$}
\newcommand{\pII}{${\rm P}_{\rm II}\;$}

\newcommand{\PIII}{${\rm P}_{\rm III}\;$}

\newcommand{\half}{
        {\lower0.00ex\hbox{\raise.6ex\hbox{\the\scriptfont0 1}
                           \kern-.5em\slash\kern-.1em\lower.45ex
                                     \hbox{\the\scriptfont0 2}}}}
\newcommand{\quarter}{
        {\lower0.00ex\hbox{\raise.6ex\hbox{\the\scriptfont0 1}
                           \kern-.5em\slash\kern-.1em\lower.45ex
                                     \hbox{\the\scriptfont0 4}}}}
\newcommand{\tquarter}{
        {\lower0.00ex\hbox{\raise.6ex\hbox{\the\scriptfont0 3}
                           \kern-.5em\slash\kern-.1em\lower.45ex
                                     \hbox{\the\scriptfont0 4}}}}
\newcommand{\eighth}{
        {\lower0.00ex\hbox{\raise.6ex\hbox{\the\scriptfont0 1}
                           \kern-.5em\slash\kern-.1em\lower.45ex
                                     \hbox{\the\scriptfont0 8}}}}
\newcommand{\othird}{
        {\lower0.00ex\hbox{\raise.6ex\hbox{\the\scriptfont0 1}
                           \kern-.5em\slash\kern-.1em\lower.45ex
                                     \hbox{\the\scriptfont0 3}}}}

\begin{document}

\title[Boundary Conditions for Scaled Random Matrix Ensembles]
{Boundary Conditions for Scaled Random Matrix Ensembles in the Bulk of the Spectrum}

\author{A.V.~Kitaev}
\address{Steklov Mathematical Institute\\Fontanka 27\\St.\,Petersburg, 191011\\Russia}
\email{\tt kitaev@pdmi.ras.ru}

\author{N.S.~Witte}
\address{Department of Mathematics and Statistics\\University of Melbourne\\Victoria 3010, Australia}
\email{\tt n.witte@ms.unimelb.edu.au}

\begin{abstract}
A spectral average which generalises the local spacing distribution of the 
eigenvalues of random $ N\times N $ hermitian matrices in the bulk of
their spectrum as $ N\to\infty $ is known to be a $\tau$-function of the 
fifth Painlev\'e system. This $\tau$-function, $ \tau(s) $, has generic parameters 
and is transcendental but is characterised by particular boundary conditions 
about the singular point $s=0$, which we determine here. When the average 
reduces to the local spacing distribution we find that $\tau$-function
is of the separatrix, or partially truncated type.
\end{abstract}

\maketitle

\section{Motivations}
\setcounter{equation}{0}
Recent studies \cite{FW_2001a,FW_2002a,FW_2004} have revealed that a generalised 
spectral average for the universality classes of
scaled random hermitian matrix ensembles - the bulk, hard edge and soft edge
classes for $ \beta=2 $ - are determined by general transcendental solutions of the 
Painlev\'e equations \PV, \PIII and \pII respectively. These works unified 
earlier studies \cite{JMMS_1980},\cite{TW_1994},\cite{TW_1994a},\cite{TW_1994b}
in the sense that this average includes local eigenvalue 
spacing distributions and moments of the characteristic polynomials as special
cases, and extended them in that the parameter sets of the Painlev\'e equations were 
exhausted.
As a consequence of this identification the logarithmic derivatives of the
spectral averages are Jimbo-Miwa-Okamoto $\sigma$-functions satisfying
second-order second-degree ordinary differential equations.
It was noted in the first mentioned works, that while identifying the precise parameters of
the Painlev\'e transcendents involved, the algebraic approach employed could not
uniquely specify adequately the boundary conditions that the particular 
solutions to the differential equations should satisfy. It is the purpose of
this paper to rectify this for the bulk scaling case and completely 
determine the boundary conditions for the differential equations of the 
$\sigma$-form.

The plan of our work is as follows - in Section 2 we review the isomonodromic 
formulation of \PVI for generic values of the parameters, give a specific parameterisation
of its monodromy matrices and the local expansion of the $ \tau$-function about
the fixed singularity $ t=0 $. In Section 3 we give the corresponding expansion
for a generalised spectral average of the circular unitary ensemble, the  
{\it spectrum singularity ensemble}, and identify
its monodromy data by comparison with the previous section. In
Section 4 we review the necessary isomonodromic formulation for \PV in a parallel
manner to the treatment of \PVI in Section 2. We take the bulk 
scaling limit of our spectrum singularity average via a formal argument in Section 5 and 
identify the resulting monodromy data applying in this case. We show that this data 
is consistent with the rigorous theory of the limit transition from \PVI to \PV as 
developed in Kitaev \cite{Ki_1994}.
We wish to advise the reader that we re-use the same symbols in the context of
the monodromy theory of \PVI as for that of \PV where there is no risk of
confusion in order not to burden the notation unnecessarily. Where both are 
discussed together then we make a notational distinction.

\section{Isomonodromy Deformation Formulation for \PVI}
\setcounter{equation}{0}
Following the conventions and notations of \cite{Ji_1982}, \cite{Ki_1994} we consider the
Lax pair of linear $2\times 2$ matrix ODEs for $ \Psi(\lambda;t) $
with four regular singularities in the $\lambda$-plane, denoted by $ \nu \in\{0,t,1,\infty\} $
\begin{equation}
   \frac{d}{d\lambda}\Psi 
  = \left( \frac{A_{0}}{\lambda}+\frac{A_{1}}{\lambda-1}+\frac{A_{t}}{\lambda-t}
      \right) \Psi , \qquad
   \frac{d}{dt}\Psi 
  = -\frac{A_{t}}{\lambda-t} \Psi .
\label{PVI_linear}
\end{equation}
We adopt the convention that the residue matrices $ A_{\nu} $ satisfy
\begin{equation}
   A_{0}+A_{t}+A_{1}=-A_{\infty}= -\frac{\theta_{\infty}}{2}\sigma_3, \quad
   \sigma_3 := \left( \begin{array}{cc} 1 & 0 \\ 0 & -1 \end{array} \right), \quad
   \theta_{\infty} \in \C\backslash\Z,
\label{PVI_Ainfty}
\end{equation}
and the constraints $ {\rm tr} A_{\nu} = 0 $, 
$ \det A_{\nu} = -\frac{1}{4}\theta^2_{\nu} $, $ \nu \in\{0,t,1\} $,
defining the formal exponents of monodromy $ \theta_{\nu} $.
The {\it $\tau$-function} for \PVI is defined as
\begin{equation}
   \frac{d}{dt}\log\tau
   = {\rm Tr}\left(\frac{A_{0}}{t}+\frac{A_{1}}{t-1} \right)A_{t} ,
\label{PVI_tau}
\end{equation}
and the {\it $\sigma$-function} as
\begin{equation}
   \zeta(t) =
   t(t-1)\frac{d}{dt}\log\tau
   +\frac{1}{4}(\theta^2_{t}-\theta^2_{\infty})t
   -\frac{1}{8}(\theta^2_{t}+\theta^2_{0}-\theta^2_{\infty}-\theta^2_{1}) ,
\label{PVI_sigma}
\end{equation}
which satisfies the second-order second degree differential equation
\begin{multline}
 \frac{d}{dt}\zeta\left(t(t-1)\frac{d^2}{dt^2}\zeta\right)^2 \\
 +\left[ 2\frac{d}{dt}\zeta\left(t\frac{d}{dt}\zeta-\zeta\right)
         -\left(\frac{d}{dt}\zeta\right)^2
         -\frac{1}{16}(\theta^2_{t}-\theta^2_{\infty})(\theta^2_{0}-\theta^2_{1})  
  \right]^2 \\
 =\left(\frac{d}{dt}\zeta+\frac{1}{4}(\theta_{t}+\theta_{\infty})^2\right)
  \left(\frac{d}{dt}\zeta+\frac{1}{4}(\theta_{t}-\theta_{\infty})^2\right) \\
  \times
  \left(\frac{d}{dt}\zeta+\frac{1}{4}(\theta_{0}+\theta_{1})^2\right)
  \left(\frac{d}{dt}\zeta+\frac{1}{4}(\theta_{0}-\theta_{1})^2\right) .
\label{PVI_SF}
\end{multline}
Furthermore we suppose that the matrices $ A_{\nu}, \nu\in \{0,t,1\} $ are 
diagonalisable, i.e. that there 
exist invertible matrices $ R_{\nu} \in SL(2,\C) $ such that
$  R^{-1}_{\nu}A_{\nu}R_{\nu} = \frac{1}{2}\theta_{\nu}\sigma_3 $,
$ \theta_{\nu} \in \C\backslash\Z $.
In the neighbourhood of a regular singularity $ \Psi(\lambda) $ can be expanded as
\begin{equation}
   \Psi(\lambda) = \sum^{\infty}_{m=0} \Psi_{m\nu}
            (\lambda-\nu)^{m+\frac{\theta_{\nu}}{2}\sigma_3} C_{\nu},
\label{PVI_regE}
\end{equation}
for $ \nu \in\{0,t,1\} $ and for $ \lambda = \infty $ in the form
\begin{equation}
   \Psi(\lambda) = \left(I+ \sum^{\infty}_{m=1} \Psi_{m\infty}\lambda^{-m} \right)
            \lambda^{-\frac{\theta_{\infty}}{2}\sigma_3}.
\label{PVI_regEinfty}
\end{equation}
The monodromy matrices $ M_{\nu} (\nu \in\{0,t,1,\infty\}) $ are defined by 
$ \left.\Psi\right|_{\nu+\delta e^{2\pi i}} = 
  \left.\Psi\right|_{\nu+\delta}M_{\nu} $,
and are given in terms of the connection matrices $ C_{\nu} $ by
\begin{equation}
   M_{\nu} = C^{-1}_{\nu}e^{\pi i\theta_{\nu}\sigma_3}C_{\nu},
   \quad C_{\infty} = I .
\label{PVI_mon}
\end{equation}
They satisfy the cyclic relation which in our conventions is taken to be 
\begin{equation}
  M_{\infty}M_{1}M_{t}M_{0}=I ,
\label{PVI_cyclic}
\end{equation}
and corresponds to the particular basis of loops given in Fig.\ref{PVI_cyclic.fig}

\begin{figure}
\begin{xy}{
   (0,0)*{};
   (60,20)="BASE";
   "BASE";"BASE"**\crv{(30,-40)&(50,-40)};?(0.25)*\dir{>};
   "BASE";"BASE"**\crv{(50,-40)&(70,-40)};?(0.25)*\dir{>};
   "BASE";"BASE"**\crv{(70,-40)&(90,-40)};?(0.25)*\dir{>};
   "BASE";"BASE"**\crv{(90,-40)&(110,-40)};?(0.25)*\dir{>};
   (47,-20)*+{\cdot\,0};(60,-20)*+{\cdot\,t};(74,-20)*+{\cdot\,1};(88,-20)*+{\cdot\,\infty};
   "BASE"+(0,3)*{\lambda_0}
}
\end{xy}
\caption{Monodromy representation of the fundamental group for 
$ \mathbb{C}\backslash\{0,t,1,\infty\} $}\label{PVI_cyclic.fig}
\end{figure}
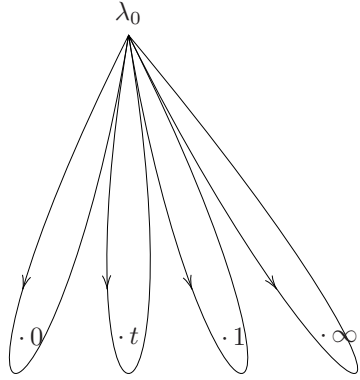

The monodromy data 
$ {\mathcal M}_{\rm VI}:=\{ \theta_{\nu}, C_{\nu}, M_{\nu} | \nu = 0,t,1,\infty  \} $
are preserved under deformations with respect to $ t $.
The invariants of the monodromy data are defined to be
$ p_{\mu} = 2\cos\pi\theta_{\mu}:= {\rm Tr}M_{\mu}, \;
  \mu \in \{0,t,1,\infty\} $ and
$ p_{\mu\nu} = 2\cos\pi\sigma_{\mu\nu}:= {\rm Tr}M_{\mu}M_{\nu}, \;
  \mu,\nu \in \{0,t,1\} $.

Now we are in a position to present the details of a parameterisation of the
monodromy matrices and the expansions of the $\tau$-function about the
singular points $ 0,1,\infty $ given in Jimbo's study \cite{Ji_1982}.
In this work Jimbo states the following conditions under which his results apply
\begin{align}
  & \theta_0, \theta_t, \theta_1, \theta_{\infty} \notin \mathbb{Z} ,
  \label{PVI_conditions:a}\\
  & 0 < {\Re}(\sigma_{0t}) < 1 ,
  \label{PVI_conditions:b}\\
  & \theta_0\pm\theta_t\pm\sigma_{0t},\quad 
    \theta_{\infty}\pm\theta_1\pm\sigma_{0t} \notin 2\mathbb{Z} ,
  \label{PVI_conditions:c}
\end{align}
and we call these {\it generic conditions}. In our application the non-resonant
condition (\ref{PVI_conditions:a}) will be adhered to, however (\ref{PVI_conditions:c})
can be relaxed in a meaningful way and we conjecture that the results of Jimbo still hold
under suitable limiting procedures where the left-hand sides $ \to 2\mathbb{Z} $.

\vfill\eject

\begin{sideways}
\begin{minipage}[c][12cm][c]{20cm}{
When $ \sigma_{0t} \neq 0 $ the parameterisation of the monodromy matrices employed 
by Jimbo is
\begin{equation}
  M_{\infty} = 
  \begin{pmatrix}
    e^{\pi i\theta_{\infty}} & 0 \cr
    0 & e^{-\pi i\theta_{\infty}}
  \end{pmatrix} ,
\label{PVI_Minfty}
\end{equation}
\begin{equation}
  M_{1} = \frac{1}{i\sin\pi\theta_{\infty}}
  \begin{pmatrix}
    \cos\pi\sigma-e^{-\pi i\theta_{\infty}}\cos\pi\theta_{1}
    &
    -2re^{-\pi i\theta_{\infty}}\sin\frac{\pi}{2}(\theta_{\infty}+\theta_{1}+\sigma)
                                  \sin\frac{\pi}{2}(\theta_{\infty}+\theta_{1}-\sigma)
  \cr
 2r^{-1}e^{\pi i\theta_{\infty}}\sin\frac{\pi}{2}(\theta_{\infty}-\theta_{1}+\sigma)
                                  \sin\frac{\pi}{2}(\theta_{\infty}-\theta_{1}-\sigma)
    &
    -\cos\pi\sigma+e^{\pi i\theta_{\infty}}\cos\pi\theta_{1}
  \end{pmatrix} ,
\label{PVI_M1}
\end{equation}
\begin{equation}
  DM_{t}D^{-1} = \frac{1}{i\sin\pi\sigma}
  \begin{pmatrix}
    e^{\pi i\sigma}\cos\pi\theta_{t}-\cos\pi\theta_{0}
    &
    -2se^{\pi i\sigma}\sin\frac{\pi}{2}(\theta_{0}+\theta_{t}-\sigma)
                                  \sin\frac{\pi}{2}(\theta_{0}-\theta_{t}+\sigma)
  \cr
 2s^{-1}e^{-\pi i\sigma}\sin\frac{\pi}{2}(\theta_{0}+\theta_{t}+\sigma)
                                  \sin\frac{\pi}{2}(\theta_{0}-\theta_{t}-\sigma)
    &
    -e^{-\pi i\sigma}\cos\pi\theta_{t}+\cos\pi\theta_{0}
  \end{pmatrix} ,
\label{PVI_Mt}
\end{equation}
\begin{equation}
  DM_{0}D^{-1} = \frac{1}{i\sin\pi\sigma}
  \begin{pmatrix}
    e^{\pi i\sigma}\cos\pi\theta_{0}-\cos\pi\theta_{t}
    &
     2s\sin\frac{\pi}{2}(\theta_{0}+\theta_{t}-\sigma)
                                  \sin\frac{\pi}{2}(\theta_{0}-\theta_{t}+\sigma)
  \cr
-2s^{-1}\sin\frac{\pi}{2}(\theta_{0}-\theta_{t}-\sigma)
                                  \sin\frac{\pi}{2}(\theta_{0}+\theta_{t}+\sigma)
    &
    -e^{-\pi i\sigma}\cos\pi\theta_{0}+\cos\pi\theta_{t}
  \end{pmatrix} ,
\label{PVI_M0}
\end{equation}
where
\begin{equation}
  D = 
  \begin{pmatrix}
    \sin\frac{\pi}{2}(\theta_{\infty}-\theta_{1}-\sigma)
    &
    r\sin\frac{\pi}{2}(\theta_{\infty}+\theta_{1}+\sigma)
    \cr
    r^{-1}\sin\frac{\pi}{2}(\theta_{\infty}-\theta_{1}+\sigma)
    &
    \sin\frac{\pi}{2}(\theta_{\infty}+\theta_{1}-\sigma)
  \end{pmatrix} ,
\label{PVI_C}
\end{equation}
with the short-hand notation $ \sigma=\sigma_{0t} $, $ s=s_{0t} $. The quantity 
$ s_{0t} $ together with $ \sigma_{0t} $ defines a unique solution to the \PVI
$\sigma$-form (\ref{PVI_SF}). The other parameter $ r $ is an arbitrary non-zero
complex constant and is a free parameter in the monodromy matrices but doesn't
appear in the local expansions. 
}
\end{minipage} 
\end{sideways}
\vfill\eject

A key formula is the following connection relation which relates $ s_{0t}, \sigma_{0t} $ 
to $ \sigma_{t1} $ and $ \sigma_{01} $.
\begin{lemma}[\cite{Ji_1982},\cite{Bo_2005}]
Under the above generic conditions (\ref{PVI_conditions:a},\ref{PVI_conditions:b},\ref{PVI_conditions:c})
and notations
\begin{multline}
  4s^{\pm 1}\sin\frac{\pi}{2}(\theta_{0}+\theta_{t}\mp\sigma_{0t})
            \sin\frac{\pi}{2}(\theta_{0}-\theta_{t}\pm\sigma_{0t}) \\ \times
            \sin\frac{\pi}{2}(\theta_{\infty}+\theta_{1}\mp\sigma_{0t})
            \sin\frac{\pi}{2}(\theta_{\infty}-\theta_{1}\pm\sigma_{0t}) \\
 = e^{\pm\pi i\sigma_{0t}}
   \left( \pm i\sin\pi\sigma_{0t}\cos\pi\sigma_{t1}-\cos\pi\theta_{t}\cos\pi\theta_{\infty}
          -\cos\pi\theta_{0}\cos\pi\theta_{1} \right) \\
          \pm i \sin\pi\sigma_{0t}\cos\pi\sigma_{01}+\cos\pi\theta_{t}\cos\pi\theta_{1}
          +\cos\pi\theta_{\infty}\cos\pi\theta_{0} .
\label{PVI_pSoln}
\end{multline}
\end{lemma}

A consequence of this relation is a constraint on the monodromy invariants which is an
algebraic variety defining a sub-manifold, the monodromy manifold, of $ \mathbb{C}^3 $.
\begin{lemma}[\cite{Ji_1982}]
The monodromy manifold for \PVI is given by
\begin{multline}
  {\mathfrak M}(p_{0t},p_{t1},p_{01}) =
  p_{0t}p_{t1}p_{01}+p_{0t}^2+p_{t1}^2+p_{01}^2
 \\
 -(p_0p_t+p_1p_{\infty})p_{0t}-(p_tp_1+p_0p_{\infty})p_{t1}-(p_0p_1+p_tp_{\infty})p_{01}
 \\
 +p_0^2+p_t^2+p_1^2+p_{\infty}^2+p_0p_tp_1p_{\infty}-4=0 .
\label{PVI_manifold}
\end{multline}
\end{lemma}

\begin{theorem}[\cite{Ji_1982}]
Under the conditions \ref{PVI_conditions:a},\ref{PVI_conditions:b},\ref{PVI_conditions:c}
then we have the asymptotic expansion of the $\tau$-function as $ t \to 0 $
in the domain $ \{t\in \C| 0<|t|<\varepsilon, |{\rm arg}(t)|<\phi\} $ for all $ \varepsilon>0 $
and any $ \phi>0 $
\begin{multline}
  \tau(t) \sim const\cdot t^{(\sigma^2-\theta^2_0-\theta^2_t)/4} \\
   \times\Bigg\{ 1 
               + \frac{(\theta^2_0-\theta^2_t-\sigma^2)(\theta^2_{\infty}-\theta^2_1-\sigma^2)}
                     {8\sigma^2}t \\
   - \hat{s}\frac{[\theta^2_0-(\theta_t-\sigma)^2][\theta^2_{\infty}-(\theta_1-\sigma)^2]}
                 {16\sigma^2(1+\sigma)^2}t^{1+\sigma} \\
   - \hat{s}^{-1}\frac{[\theta^2_0-(\theta_t+\sigma)^2][\theta^2_{\infty}-(\theta_1+\sigma)^2]}
                 {16\sigma^2(1-\sigma)^2}t^{1-\sigma}
   + {\rm O}(|t|^{2(1-\Re(\sigma))}) \Bigg\} ,
\label{PVI_tExp_0}
\end{multline}
where $ \sigma \neq 0 $ and $ \hat{s} $ are related to $ s $ through
\begin{multline}
  \hat{s} = s \frac{\Gamma^2(1-\sigma)\Gamma(1+\frac{1}{2}(\theta_0+\theta_t+\sigma))
                    \Gamma(1+\frac{1}{2}(-\theta_0+\theta_t+\sigma))}
                   {\Gamma^2(1+\sigma)\Gamma(1+\frac{1}{2}(\theta_0+\theta_t-\sigma))
                    \Gamma(1+\frac{1}{2}(-\theta_0+\theta_t-\sigma))} \\
        \times\frac{\Gamma(1+\frac{1}{2}(\theta_{\infty}+\theta_1+\sigma))
                    \Gamma(1+\frac{1}{2}(-\theta_{\infty}+\theta_1+\sigma))}
                   {\Gamma(1+\frac{1}{2}(\theta_{\infty}+\theta_1-\sigma))
                    \Gamma(1+\frac{1}{2}(-\theta_{\infty}+\theta_1-\sigma))} ,
\label{PVI_s_0t}
\end{multline}
and we employ the short-hand notation $ s=s_{0t} $,$ \hat{s}=\hat{s}_{0t} $ and 
$ \sigma=\sigma_{0t} $. The monodromy data defining a unique solution to the sixth
Painlev\'e system is $ \{\sigma_{0t}, s_{0t}\} $.
\end{theorem}

\section{The Spectrum Singularity Ensemble}
\setcounter{equation}{0}
A fundamental ensemble in random matrix theory is the ensemble of finite rank 
(rank $ =N $) random unitary matrices - the {\it Dyson circular unitary ensemble (CUE)}.
Consider a member of this ensemble $ U $ with eigenvalues 
$ z_1=e^{i\theta_l}, \ldots, z_N=e^{i\theta_N} $. Then the eigenvalue probability 
density function for the ensemble is the Haar measure for $ U(N) $
\begin{equation}
   p_N(\theta_1,\ldots,\theta_N) := \frac{1}{(2\pi)^N N!} 
       \prod_{1 \leq j < k \leq N} |z_j-z_k|^2 .
\label{CUE}
\end{equation}
A generalisation of this ensemble, referred to as the 
{\it spectrum singularity ensemble (SSE)} \cite{FW_2004},
has an eigenvalue probability density 
function containing additional algebraic singularities at $ z = 0, -1, -1/t $ where
$ t = e^{i\phi}, \phi \in [0,2\pi) $.
In the log-gas picture of the CUE the eigenvalues of a random unitary matrix 
are mobile unit charges subject to a pair-wise mutual logarithmic repulsion and
in the SSE these charges are subject to additional external fields -
the logarithmic electrostatic potential of impurity charges located at the ends 
of the sector, at $ z = -1, -1/t $ with charges $ \omega_1, \mu $ 
respectively, and an external electric field of strength $ \omega_2 $.
The generalised spectral average we wish to investigate is the partition function 
of such an electrostatic system, and is given by the $N$-dimensional integral \cite{FW_2004}
\begin{multline}
   \mathcal{A}_N(t;\omega_1,\omega_2,\mu;\xi^*)
    := \frac{1}{N!} 
       \left(\int^{\pi}_{-\pi}-\xi^*\int^{\pi}_{\pi-\phi}\right) \frac{d\theta_1}{2\pi}
\ldots \left(\int^{\pi}_{-\pi}-\xi^*\int^{\pi}_{\pi-\phi}\right) \frac{d\theta_N}{2\pi} \\
      \times\prod^{N}_{l=1} z^{-i\omega_2}_l |1+z_l|^{2\omega_1} |1+tz_l|^{2\mu}
      \prod_{1 \leq j < k \leq N} |z_j-z_k|^2,
\label{SSE_AN}
\end{multline}
where $ \xi^* \in \C $ and the 
parameters $ \omega_1,\omega_2,\mu \in \C $, $ \omega = \omega_1+i\omega_2 $,
are restricted with $ \Re(2\omega_1), \Re(2\mu) > -1 $, $ N \in \Z_{\geq 0} $. 
The independent variable $ t $, whilst defined on the unit circle $ |t| = 1 $, can 
be analytically continued into the cut complex $t$-plane. 

It was shown in \cite{FW_2004} that the average $ \mathcal{A}_N(t) $ is the 
$N\in \Z_{\geq 0}$th member of a sequence of classical $\tau$-functions of the sixth 
Painlev\'e system. Thus this average is characterised by solutions of the nonlinear 
differential equation (\ref{PVI_SF}) subject to the boundary conditions 
given in \cite{FW_2006a}, which were derived following an idea introduced in \cite{FW_2006}.

\begin{theorem}\cite{FW_2006a}
For generic values of the parameters $ \mu,\omega,\bar{\omega} $, subject to 
$ 2\mu+2\omega_1 \notin \Z $, $ \Re(2\mu+2\omega_1)>0 $
the spectral average $ \mathcal{A}_N $ has the following expansion about $ t=1 $ 
\begin{multline}
  \mathcal{A}_N(t) =
  \prod^{N-1}_{k=0}\frac{k!\Gamma(2\mu+2\omega_1+k+1) }{ 
                      \Gamma(1+k+\mu+\omega)\Gamma(1+k+\mu+\bar{\omega})}
  \Bigg\{ 1 \\ 
  + \frac{N\mu(\bar{\omega}-\omega)}{2\mu+2\omega_1}(1-t) + {\rm O}((1-t)^2) \\
  + \frac{(-1)^{N+1}}{\sin\pi(2\mu+2\omega_1)}
    \left(\xi^*\frac{e^{-\pi i(\mu-\bar{\omega})}}{2i}
            +\frac{\sin\pi 2\mu\sin\pi(\mu+\omega)}{\sin\pi(2\mu+2\omega_1)} \right) \\
  \times
  \frac{\Gamma(1+2\mu)\Gamma(1+2\omega_1)\Gamma(1+\mu+\omega)\Gamma(1+\mu+\bar{\omega})
        }{ 
   \Gamma^2(2\mu+2\omega_1+2)\Gamma(2\mu+2\omega_1+1)\Gamma(N)\Gamma(-N-2\mu-2\omega_1) }
         (1-t)^{1+2\mu+2\omega_1} \\
  \times\left( 1+{\rm O}(1-t)\right) + {\rm O}((1-t)^{2+4\mu+4\omega_1})
  \Bigg\} .
\label{AN_exp_1}
\end{multline}
\end{theorem}

The precise relationship between the spectrum singularity average $ \mathcal{A}_N(t) $ 
and the isomonodromy theory of the sixth Painlev\'e system was first given in 
\cite{FW_2006a}.
However because we intend to apply the Kitaev theory \cite{Ki_1994}, which treats 
the coalescence of the regular singularities $ 0,t $ of the \PVI system under the 
transition limit $ t\to 0 $, and the natural bulk scaling of $ A_N(t) $ is through 
the limit $ t\to 1 $ the monodromy data of our application will have to be recast 
in a suitable form.
\begin{theorem}\cite{FW_2006a}\label{SSE}
The spectrum singularity ensemble can be defined by the following monodromy data. 
The formal monodromy exponents are
\begin{equation}
   \theta_{0} = N+2\mu ,\quad
   \theta_{t} =-N-2\omega_1 ,\quad
   \theta_{1} = -\mu-\omega ,\quad 
   \theta_{\infty} = \mu+\bar{\omega} ,
   \label{SSE_Mexp:A}
\end{equation}
with the necessary classical condition
$ \theta_{0}-\theta_{t}+\theta_{1}-\theta_{\infty} \in 2\Z $. 
The relevant monodromy invariant is
\begin{equation}
   \sigma_{0t}=2\mu+2\omega_1 .
\label{SSE_MI}
\end{equation}
In this case the associated monodromy coefficient is given by 
$ s_{0t}=\epsilon\pi\hat{s}/2 $ under the limiting process 
$ \epsilon := \sigma_{0t}+\theta_{1}-\theta_{\infty} \to 0 $,
along with
\begin{equation}
  \hat{s} \sin\pi2\omega_1\sin\pi(\mu+\omega)
  = \frac{\sin\pi 2\mu\sin\pi(\mu+\omega)}{\sin\pi(2\mu+2\omega_1)}
       +\xi^*\frac{e^{-\pi i(\mu-\bar{\omega})}}{2i} .
  \label{SSE_st1:A}
\end{equation}
All monodromy matrices are upper triangular
\begin{gather}
  M_0 = \begin{pmatrix}
          e^{-\pi i\theta_{0}} & m_0 \\ 0 & e^{\pi i\theta_{0}} \\
        \end{pmatrix} ,
  \label{SSE_MA:a}\\
  M_t = \begin{pmatrix}
          e^{\pi i\theta_{t}} & m_t \\  0 & e^{-\pi i\theta_{t}} \\
        \end{pmatrix} ,
  \label{SSE_MA:b}\\
  M_1 = \begin{pmatrix}
          e^{-\pi i\theta_{1}} & m_1 \\ 0 & e^{\pi i\theta_{1}} \\
        \end{pmatrix} ,
  \label{SSE_MA:c}
\end{gather}
with upper triangular elements of the form
\begin{gather}
  m_0 =(-1)^N2ir\frac{\sin\pi2\mu}{\sin^2\pi(2\mu+2\omega_{1})}
  \bigg\{ -\frac{r}{\hat{s}}\sin\pi(\mu+\bar{\omega})
          +\sin\pi(2\mu+2\omega_{1})\sin\pi(\mu+\omega) \bigg\} ,
  \label{SSE_MA:d}\\
  m_t =(-1)^N2ir\frac{1}{\sin^2\pi(2\mu+2\omega_{1})}
  \bigg\{ \frac{r}{\hat{s}}e^{-\pi i(2\mu+2\omega_{1})}\sin\pi 2\mu\sin\pi(\mu+\bar{\omega})
  \label{SSE_MA:e}\\
      \phantom{(-1)^N2ir\frac{1}{\sin^2\pi(2\mu+2\omega_{1})}}
          +\sin\pi(2\mu+2\omega_{1})\sin\pi 2\omega_1\sin\pi(\mu+\omega) \bigg\} ,
  \nonumber\\
  m_1 =-2ire^{-\pi i(\mu+\bar{\omega})}\sin\pi(\mu+\omega) ,
  \label{SSE_MA:f}
\end{gather}
where $ r $ is a non-zero arbitrary constant.
\end{theorem}
\begin{proof}
In \cite{FW_2006a} it was noted that the monodromy exponents could be given by any 
one of three sets
\begin{equation}
	\{\theta_0,\theta_t,\theta_1,\theta_{\infty}\} =
	\begin{cases} 
	        N+2\mu,N+2\omega_1,\mu+\omega,\mu+\bar{\omega} \\
	        N,N+2\mu+2\omega_1,\mu-\omega,\mu-\bar{\omega} \\
	        N+\mu+\omega,N+\mu+\bar{\omega},-2\mu,2\omega_1 
	\end{cases} ,
\end{equation}
modulo permutations of the exponents, an even number of sign reversals and subject to
some constraints which we elaborate on later. It was found that the first set lead to
upper triangular monodromy matrices, the second to full monodromy matrices except for
one which was a multiple of the identity and the third to lower triangular matrices.
The second set strictly violates the condition of non-resonant monodromy exponents 
and the third set is essentially equivalent to the first and therefore we choose an
example from the first set. These constraints, when applied to the first set,
imply that our choices are now
\begin{equation}
	\{\theta_0,\theta_t,\theta_1,\theta_{\infty}\} =
	\begin{cases} 
	       \pm(\mu+\omega),\pm(N+2\mu),\pm(N+2\omega_1),\pm(\mu+\bar{\omega}) \\
	       \pm(\mu+\bar{\omega}),\pm(N+2\omega_1),\pm(N+2\mu),\pm(\mu+\omega) \\
	       \pm(\mu+\omega),\pm(N+2\mu),\pm(N+2\omega_1),\pm(\mu+\bar{\omega}) \\
	       \pm(\mu+\bar{\omega}),\pm(N+2\omega_1),\pm(N+2\mu),\pm(\mu+\omega)
	\end{cases} ,
\label{SSE_choice}
\end{equation}
with an even number of sign reversals. The monodromy invariants were given as
$ \sigma_{0t}=N-\mu+\bar{\omega} $,
$ \sigma_{t1}=2\mu+2\omega_1 $,
$ \sigma_{01}=N-\mu+\omega  $.

Next we employ the linear fractional transformation of our original system
$ t \mapsto 1-t $, $ \theta_0 \leftrightarrow \theta_1 $,
$ \sigma_{0t} \mapsto \sigma_{t1} $, $ \sigma_{t1} \mapsto \sigma_{0t} $,
and it is this new system to which we shall refer to henceforth.
We take one choice from the above (\ref{SSE_choice}) which is applicable to the theory
given in \cite{Ki_1994} and this is (\ref{SSE_Mexp:A}) and (\ref{SSE_MI}).
Having made this choice we observe that in order for the Jimbo parameterisation
given in (\ref{PVI_tExp_0}) and (\ref{PVI_s_0t}) to be consistent with (\ref{AN_exp_1})
under the map $ t\mapsto 1-t $ then $ s_{0t} $ must vanish in the manner given in
Theorem (\ref{SSE}). The resulting finite coefficient $ \hat{s} $ is then given by formula 
(\ref{SSE_st1:A}). It can be verified also that the overall prefactors and the leading
order analytic terms in both expansions agree precisely with the choices we have
made. The monodromy matrices are then given by (\ref{SSE_MA:a}-\ref{SSE_MA:f}) using
the parameterisation in (\ref{PVI_M1},\ref{PVI_Mt},\ref{PVI_M0},\ref{PVI_C}).
The off diagonal elements satisfy the relation
$ e^{-\pi i\theta_{0}}m_{0}+e^{-\pi i\theta_{t}}m_{t}
 +e^{-\pi i(\theta_{0}+\theta_{t}-\theta_{\infty})}m_{1} = 0 $, as required by the
cyclic identity (\ref{PVI_cyclic}).
\end{proof}

\section{Isomonodromy Deformation Formulation for \PV}
\setcounter{equation}{0}
The work of Jimbo \cite{Ji_1982} has a slightly different formulation of the
isomonodromy problem for \PV from that of Kitaev and collaborators 
\cite{Ki_1994,AK_1997,AK_2000} and it is the latter form that we adopt.
We formulate the \PV isomonodromic system as the Lax pair of linear $ 2\times 2 $ 
matrix ODEs for
$ \Psi(\lambda;t) $ 
with two regular singular points at $ \nu = 0,1 $ and an irregular one at
$ \infty $ with Poincare rank unity
\begin{equation}
   \frac{d}{d\lambda}\Psi 
  = \left( \frac{t}{2}\sigma_3+\frac{A_{0}}{\lambda}+\frac{A_{1}}{\lambda-1}
      \right) \Psi, \qquad
   \frac{d}{dt}\Psi 
  = \left\{ \frac{\lambda}{2}\sigma_3
              +\frac{1}{t}\left(\frac{\theta_{\infty}}{2}\sigma_3+A_{0}+A_{1}\right)
      \right\} \Psi.
\label{PV_linear}
\end{equation}
We require that the residue matrices satisfy
$ {\rm diag}(A_{0}+A_{1}) = -\frac{\theta_{\infty}}{2}\sigma_3 $,
$ \theta_{\infty} \in \C\backslash\Z $,
and the constraints
$ {\rm tr} A_{\nu} = 0 $, $ \det A_{\nu} = -\frac{1}{4} \theta^2_{\nu} $,
$ \nu \in\{0,1\} $.
Again we assume that there exists invertible matrices $ R_{\nu} \in SL(2,\C) $ such that
$ R^{-1}_{\nu}A_{\nu}R_{\nu} = \frac{1}{2}\theta_{\nu}\sigma_3 $,
$ \theta_{\nu} \in \C\backslash\Z $, $ \nu \in\{0,1\} $.

Within the works of Kitaev and collaborators there are two different conventions 
for the monodromy data
employed and we will make the distinction between the two by using a car\'et 
for those of \cite{Ki_1994} as opposed to those of \cite{AK_1997,AK_2000}.
In the neighbourhood of the irregular singularity the canonical solutions 
$ \Psi^k(\lambda) $ can be expanded as
\begin{equation}
   \Psi^k(\lambda) \sim \Big(I+ \sum^{\infty}_{m=1} \Psi^k_{m\infty}\lambda^{-m} \Big)
   \exp\left\{ \big(\frac{t}{2}\lambda - \frac{1}{2}\theta_{\infty}\log\lambda 
               \big)\sigma_3 \right\}.
\label{PV_irregE}
\end{equation}
in the sectors $ -3\pi/2 + \pi k < \arg \lambda t < \pi/2 + \pi k, \; k = 1,2 $.
The Stokes matrices $ S_k $ are defined as
$ \Psi^{k+1}(\lambda) = \Psi^k(\lambda) \hat{S}_k $,
which have the structure
\begin{equation}
  \hat{S}_{2l} = \left( \begin{array}{cc} 1 & 0 \\ \hat{s}_{2l} & 1 \end{array} \right),
  \quad
  \hat{S}_{2l+1} = \left( \begin{array}{cc} 1 & \hat{s}_{2l+1} \\ 0 & 1 \end{array} \right),
\label{PV_Stokess}
\end{equation}
where $ \hat{s}_k $ are the Stokes multipliers.
The monodromy matrix $ \hat{M}_{k\infty} $ is given by
$ \hat{M}_{k\infty} = \hat{S}_{k}\hat{S}_{k+1}e^{\pi i\theta_{\infty}\sigma_3} $,
$ C_{\infty} = I $, and we set $ \hat{M}_{\infty} =  \hat{M}_{0\infty} $. 
The monodromy matrices at the 
regular singularities are defined in the same way as in (\ref{PVI_mon}). The
cyclic relation in this case is 
\begin{equation}
  \hat{M}_{0}\hat{M}_{1}\hat{M}_{\infty} = I , 
\label{PV_cyclicH}
\end{equation}
reversed in order from the usual convention because of the nature of the limiting 
transition we are to consider later. The monodromy data preserved here are 
$ {\mathcal M}_{\rm V}:=\{ \theta_{\nu}, C_{\nu}, \hat{M}_{\nu}, \hat{s}_1, \hat{s}_2 \} $.
If we define the monodromy invariant 
$ 2\cos\pi\sigma = {\rm Tr}(\hat{M}_{0}\hat{M}_{1}) $,
then the following constraint, analogous to (\ref{PVI_manifold}), applies
\begin{equation}
   \hat{s}_0\hat{s}_1e^{-\pi i\theta_{\infty}} 
   = 4\sin\frac{\pi}{2}(\theta_{\infty}+\sigma)\sin\frac{\pi}{2}(\theta_{\infty}-\sigma) .
\label{PV_manifold}
\end{equation}
In contrast the cyclic relation adopted in \cite{AK_1997,AK_2000} is
\begin{equation}
  M_{\infty}M_{1}M_{0} = I , 
\label{PV_cyclic}
\end{equation}
where
\begin{equation} 
  M_{\infty} \equiv M_{2\infty} = S_2 e^{\pi i\theta_{\infty}\sigma_3}S_1 .
\label{PV_MInf}
\end{equation}
The inter-relationship between the two sets of monodromy data is given by
\begin{gather}
  \hat{S}_0 = S_1, \qquad \hat{S}_1 = S_2 \\
  \hat{M}_0 = S_{1}M_{1}M_{0}M_{1}^{-1}S_{1}^{-1}, \quad 
  \hat{M}_1 = S_{1}M_{1}S_{1}^{-1}, \quad 
  \hat{M}_{\infty} = S_{1}M_{\infty}S_{1}^{-1} .
\label{PV_Mxfm}
\end{gather}

The $\sigma$-function in the Jimbo formulation of the \PV linear system is equivalent
to that of Andreev and Kitaev \cite{AK_2000} and is defined as
\begin{equation}
   \zeta(t) =
   t\frac{d}{dt}\log\tau
   +\frac{1}{2}(\theta_{0}+\theta_{\infty})t
   +\frac{1}{4}[(\theta_{0}+\theta_{\infty})^2-\theta^2_{1}] ,
\label{PV_sigma}
\end{equation}
which satisfies the second-order second degree differential equation
\begin{multline}
  \left(t\frac{d^2}{dt^2}\zeta\right)^2
 =\left[ \zeta-t\frac{d}{dt}\zeta+2\left(\frac{d}{dt}\zeta\right)^2
         -(2\theta_{0}+\theta_{\infty})\frac{d}{dt}\zeta
  \right]^2 \\
 -4\frac{d}{dt}\zeta
   \left(\frac{d}{dt}\zeta-\theta_{0}\right)
   \left(\frac{d}{dt}\zeta-\frac{1}{2}(\theta_{0}-\theta_{1}+\theta_{\infty})\right)
   \left(\frac{d}{dt}\zeta-\frac{1}{2}(\theta_{0}+\theta_{1}+\theta_{\infty})\right) .
\label{PV_SF}
\end{multline}
The asymptotic expansion of the $\tau$-function is given by the following theorem.
\begin{theorem}[\cite{Ji_1982}]\label{PV_Exp}
Under the conditions
$ \theta_0, \theta_1 \notin \mathbb{Z} $,
$ 0 < \Re(\sigma) < 1 $ and
$ \theta_1\pm\theta_0\pm\sigma, \theta_{\infty}\pm\sigma \notin 2\mathbb{Z} $
then we have the asymptotic expansion of the $\tau$-function as $ t \to 0 $
\begin{multline}
  \tau(t) \sim const\cdot t^{(\sigma^2-\theta^2_{\infty})/4}
       \Bigg\{ 1 
               - \frac{\theta_{\infty}(\theta^2_1-\theta^2_0+\sigma^2)}
                     {4\sigma^2}t \\
   - \hat{s}\frac{[\theta_{\infty}-\sigma][\theta^2_0-(\theta_1-\sigma)^2]}
                 {8\sigma^2(1+\sigma)^2}t^{1+\sigma} \\
   - \hat{s}^{-1}\frac{[\theta_{\infty}+\sigma][\theta^2_0-(\theta_1+\sigma)^2]}
                 {8\sigma^2(1-\sigma)^2}t^{1-\sigma}
   + {\rm O}(|t|^{2(1-\Re(\sigma))}) \Bigg\} ,
\label{PV_tExp_0}
\end{multline}
where $ \sigma \neq 0 $ and $ \hat{s} $ are related to $ s $ through
\begin{multline}
  \hat{s} = s \frac{\Gamma^2(1-\sigma)\Gamma(1+\frac{1}{2}(\theta_1+\theta_0+\sigma))
                    \Gamma(1+\frac{1}{2}(\theta_1-\theta_0+\sigma))}
                   {\Gamma^2(1+\sigma)\Gamma(1+\frac{1}{2}(\theta_1+\theta_0-\sigma))
                    \Gamma(1+\frac{1}{2}(\theta_1-\theta_0-\sigma))} \\
        \times\frac{\Gamma(1+\frac{1}{2}(\theta_{\infty}+\sigma))}
                   {\Gamma(1+\frac{1}{2}(\theta_{\infty}-\sigma))} .
\label{PV_s_0t}
\end{multline}
\end{theorem}
\vfill\eject

\begin{sideways}
\begin{minipage}[c][12cm][c]{20cm}{
When $ \sigma \neq \Z $ we have the parameterisation of the \PV monodromy matrices employed 
in \cite{Ji_1982}, and in Theorem 6.2 of \cite{AK_2000}
\begin{equation}
  DM_{0}D^{-1} = \frac{1}{i\sin\pi\sigma}
  \begin{pmatrix}
    e^{\pi i\sigma}\cos\pi\theta_{0}-\cos\pi\theta_{1}
    &
    -2rse^{-\pi i\sigma}\sin\frac{\pi}{2}(\theta_{1}-\theta_{0}-\sigma)
                                  \sin\frac{\pi}{2}(\theta_{1}+\theta_{0}-\sigma)
  \cr
 2(rs)^{-1}e^{\pi i\sigma}\sin\frac{\pi}{2}(\theta_{1}-\theta_{0}+\sigma)
                                  \sin\frac{\pi}{2}(\theta_{1}+\theta_{0}+\sigma)
    &
    -e^{-\pi i\sigma}\cos\pi\theta_{0}+\cos\pi\theta_{1}
  \end{pmatrix} ,
\label{PV_M0}
\end{equation}
\begin{equation}
  DM_{1}D^{-1} = \frac{1}{i\sin\pi\sigma}
  \begin{pmatrix}
    e^{\pi i\sigma}\cos\pi\theta_{1}-\cos\pi\theta_{0}
    &
     2rs\sin\frac{\pi}{2}(\theta_{1}-\theta_{0}-\sigma)
                                  \sin\frac{\pi}{2}(\theta_{0}+\theta_{1}-\sigma)
  \cr
-2(rs)^{-1}\sin\frac{\pi}{2}(\theta_{1}-\theta_{0}+\sigma)
                                  \sin\frac{\pi}{2}(\theta_{0}+\theta_{1}+\sigma)
    &
    -e^{-\pi i\sigma}\cos\pi\theta_{1}+\cos\pi\theta_{0}
  \end{pmatrix} ,
\label{PV_M1}
\end{equation}
where
\begin{equation}
  D = 
  \begin{pmatrix}
    e^{-\pi i\sigma/2}
    &
    r\sin\frac{\pi}{2}(\theta_{\infty}+\sigma)
    \cr
    r^{-1}e^{\pi i\sigma/2}
    &
    \sin\frac{\pi}{2}(\theta_{\infty}-\sigma)
  \end{pmatrix} .
\label{PV_D}
\end{equation}
The Stokes multipliers are given by the formulae
\begin{equation} 
  s_1 = -
  \frac{2\pi ir^{-1}}{\Gamma(1-\frac{\sigma-\theta_{\infty}}{2})\Gamma(\frac{\sigma+\theta_{\infty}}{2})} ,
  \qquad
  s_2 = -e^{\pi i\theta_{\infty}}
  \frac{2\pi ir}{\Gamma(1-\frac{\sigma+\theta_{\infty}}{2})\Gamma(\frac{\sigma-\theta_{\infty}}{2})} ,
\label{PV_StokesM}
\end{equation} 
The parameter $ r $ is again an arbitrary non-zero complex constant and doesn't 
appear in the expansion formulae.
}
\end{minipage} 
\end{sideways}
\vfill\eject

\section{The Bulk Scaling Ensemble}
\setcounter{equation}{0}
Also studied in \cite{FW_2004} was the scaling limit of the spectrum singularity
ensemble in the neighbourhood of the singularity at $ \phi = 0 $ ($ t=1 $), 
to the bulk regime via the limit $ N \to \infty $. It was shown there that the 
log-derivative of the average $ \mathcal{A}_N(t) $ converged to a solution of the 
Jimbo-Miwa-Okamoto 
$\sigma$-function for the fifth Painlev\'e equation. Thus the most general 
universality class of hermitian random matrix ensembles in the bulk scaling limit 
was found to be a generic four parameter class - three arbitrary parameters, 
$ \omega_1,\omega_2,\mu $,
appearing in the differential equation and one, $ \xi^* $, in the boundary data.
Thus this class is characterised by a transcendental solution to a generic fifth
Painlev\'e equation. 

The establishment of this result was based on a formal scaling limit of the 
\PVI second-order second-degree ordinary differential equation (\ref{PVI_SF}) 
to the corresponding \PV ODE (\ref{PV_SF}). 
Defining the scaling variables
\begin{equation}
  t:=e^{-x/N}, \quad
  u(x;\omega_1,\omega_2,\mu;\xi^*) := x\frac{d}{dx} \lim_{N \to \infty} 
            \log\mathcal{A}_N(t;\omega_1,\omega_2,\mu;\xi^*),
\label{SSE_bulk}
\end{equation}
it was found that $ u(x) $ is related to a solution of an 
alternative Jimbo-Miwa-Okamoto 
$\sigma$-form of the fifth Painlev\'e equation $ h_{\V}(s) $ by
\begin{equation}
   u(ix;\omega_1,\omega_2,\mu) = h_{\V}(x;{\bf v})
   + (-\frac{1}{4} \omega - \frac{3}{4} \bar{\omega}+\mu)x
     + \eighth (\omega-\bar{\omega})^2-\mu(\omega+\bar{\omega}),
\label{Bulk_sigma}
\end{equation}
with the Okamoto parameters
\begin{equation}
   (v_1,v_2,v_3,v_4) = 
    ( \omega_1-\frac{1}{2} i\omega_2, -\omega_1-\frac{1}{2} i\omega_2,
               \mu+\frac{1}{2} i\omega_2, -\mu+\frac{1}{2} i\omega_2 ).
\label{Bulk_param}
\end{equation}
The alternative Jimbo-Miwa-Okamoto $\sigma$-form of the fifth Painlev\'e 
equation (\ref{PV_SF}) is
\begin{equation}
  (xh_{\V}'')^2-[h_{\V}-xh_{\V}'+2(h_{\V}')^2]^2
  +4 \prod^4_{k=1}(h_{\V}'+v_k) = 0,
\label{JMO_PV}
\end{equation}
where the constraint $ v_1+v_2+v_3+v_4=0 $ applies.
The scaled spectral average is a $\tau$-function for this system and related
to $ u(x) $ by
\begin{equation}
   \mathcal{A}(x) = \exp \int^x_0 \frac{dy}{y} u(y;\omega_1,\omega_2,\mu) . 
\end{equation}

We have the following expansion of $ \mathcal{A}(x) $ about $x=0$ under the 
above formal bulk scaling limit.
\begin{theorem}\label{Bulk_scaling}
Under the conditions 
$ \Re(2\mu), \Re(2\omega_1), \Re(\mu+\omega), \Re(\mu+\bar{\omega}) >-1 $, 
$ 0 \leq \Re(2\mu+2\omega_1) <1 $, 
the spectral average $ \mathcal{A}_N(t;\omega_1,\omega_2,\mu;\xi^*) $
has a bulk scaling limit with $ t\mapsto \exp(-x/N) $ as $ N \to \infty $
\begin{equation}
  \prod^{N-1}_{k=0}\frac{\Gamma(1+k+\mu+\omega)\Gamma(1+k+\mu+\bar{\omega})
                        }{ k!\Gamma(2\mu+2\omega_1+k+1)} 
   \mathcal{A}_N(t;\omega_1,\omega_2,\mu;\xi^*)
  \mathop{\sim}\limits_{N \to \infty} \mathcal{A}(x) ,
\label{SSE_limit}
\end{equation}
and the scaled $\mathcal{A}$-function has the following expansion as $ x \to 0 $
\begin{multline}
  \mathcal{A}(x) =
  1 + \frac{\mu(\bar{\omega}-\omega)}{2\mu+2\omega_1}x+{\rm O}(x^2) \\
  + \frac{1}{\pi}\left(\xi^*\frac{e^{-\pi i(\mu-\bar{\omega})}}{2i}
            +\frac{\sin\pi 2\mu\sin\pi(\mu+\omega)}{\sin\pi(2\mu+2\omega_1)} \right) \\
  \times
  \frac{\Gamma(1+2\mu)\Gamma(1+2\omega_1)\Gamma(1+\mu+\omega)\Gamma(1+\mu+\bar{\omega})
       }{\Gamma^2(2\mu+2\omega_1+2)\Gamma(2\mu+2\omega_1+1)}
         x^{1+2\mu+2\omega_1}\left(1+{\rm O}(x)\right) \\
  +{\rm O}(x^{2+4\mu+4\omega_1}) .
\label{Bulk_tau}
\end{multline}
\end{theorem}
\begin{proof}
The expansion (\ref{AN_exp_1}) is valid uniformly in $ N $ under the substitution
$ t \mapsto \exp(-x/N) $ and by taking the limit $ N \to \infty $ and employing
the asymptotic formula for the ratio of gamma functions we arrive at
(\ref{Bulk_tau}). 
\end{proof}

Comparison of this result with the general Theorem \ref{PV_Exp} of Jimbo leads to 
the following conclusions.
\begin{theorem}\label{Bulk_x=0Mdata}
The \PV monodromy parameters for the bulk scaled spectrum singularity ensemble
are
\begin{equation}
  \theta_{0} = \mu+\bar{\omega}, \quad
  \theta_{1} = -\mu-\omega, \quad
  \theta_{\infty} = 2\mu-2\omega_1 ,
  \label{Bulk_Mexp:A}
\end{equation}
with the invariant $ \sigma = 2\mu+2\omega_1 $. The monodromy coefficient is 
\begin{equation}
  s = -\frac{\sin\pi2\mu}{\sin\pi2\omega_1}
      -\frac{\sin\pi(2\mu+2\omega_1)}{\sin\pi2\omega_1\sin\pi(\mu+\omega)}
       \xi^*\frac{e^{-\pi i(\mu-\bar{\omega})}}{2i} .
\label{Bulk_s:A}
\end{equation}
\end{theorem}
\begin{proof}
The identification of (\ref{JMO_PV}) with the parameter set (\ref{Bulk_param})
with (\ref{PV_SF}) leads to the following solution sets for the \PV formal
monodromy exponents
\begin{multline}
    \{  \frac{1}{2}\theta_{0}+\frac{1}{4}\theta_{\infty},
       -\frac{1}{2}\theta_{0}+\frac{1}{4}\theta_{\infty},
        \frac{1}{2}\theta_{1}-\frac{1}{4}\theta_{\infty},
       -\frac{1}{2}\theta_{1}-\frac{1}{4}\theta_{\infty} \} \\
 =  \{  \mu-\frac{1}{2}i\omega_2,-\mu-\frac{1}{2}i\omega_2,
        \omega-\frac{1}{2}i\omega_2, -\bar{\omega}-\frac{1}{2}i\omega_2 \},
\end{multline} 
modulo permutations. One choice, consistent with subsequent calculations, is given
in (\ref{Bulk_Mexp:A}). Using this choice and comparing the two asymptotic expansions 
(\ref{Bulk_tau}) and (\ref{PV_tExp_0}) through their relationship
\begin{equation}
  \mathcal{A}(x) = Ce^{\frac{1}{4}(\theta_{\infty}-\omega+\bar{\omega})x}
          x^{\frac{1}{4}[(\theta_0+\theta_{\infty})^2-\theta_{1}^2]
                         -\frac{1}{2}(\theta_0+\frac{1}{2}\theta_{\infty})^2
                         -2\mu\omega_1+\frac{1}{8}(\omega-\bar{\omega})^2}\tau(x) .
\end{equation}
we see that $ \sigma^2 = (2\mu+2\omega_1)^2 $ from consideration of the algebraic
prefactor. The coefficients of the analytic terms then agree upon using this value
for the exponent. Finally the coefficient of the non-analytic 
$ x^{1-\sigma} $ vanishes because $ \theta_0-\theta_1-\sigma=0 $ and
(\ref{Bulk_s:A}) follows from a matching of the remaining term.
\end{proof}

\section{The Limit transition \PVI to \PV}
\setcounter{equation}{0}
The limit transition from \PVI to \PV is from one monodromy preserving 
system to another and it is known \cite{Bo_1994} that there exists no continuous 
transition between such systems which is itself monodromy preserving. 
However there are discrete transitions which are and one example is a sequence of 
even parity Schlesinger transformations of the \PVI system, i.e. 
$ \theta_{\nu} \mapsto \theta_{\nu}+2n, n\in\Z^+ $. 
In a sense the \PV system arises then as a fixed point of this map although the
deformation variable also must scale in an appropriate way.
This approach was fundamental
to the study of Kitaev \cite{Ki_1994} who considered an example which is 
relevant to the present work. By way of comparison with our Theorem \ref{Bulk_scaling} 
let us now consider the formal scaling undertaken in limit II
of \PVI as given in \cite{Ki_1994}. In this limit the \PVI formal exponents of
monodromy are related to the \PV exponents by
\begin{align}
  \theta_{t{\VI}} & = \theta_{6}-2n =: -\frac{1}{\epsilon},
  \label{iReg_coal_theta:a}\\
  \theta_{0{\VI}} & = \theta_{\infty {\V}}-\theta_{6}+2n 
                = \theta_{\infty {\V}}+\frac{1}{\epsilon},
  \label{iReg_coal_theta:b}\\
  \theta_{1{\VI}} & = \theta_{1{\V}},
  \label{iReg_coal_theta:c}\\
  \theta_{\infty {\VI}} & = \theta_{0{\V}},
  \label{iReg_coal_theta:d}
\end{align}
and the limit $ n \to \infty $ or $ \epsilon \to 0 $ is taken. The quantity 
$ \theta_{6} $ is an additional constant which will fixed in our application. 
The two regular singularities $ \nu = 0,t $ coalesce through
$ t_{\VI} = \epsilon t_{\V} = O(\epsilon) $ and the transition from the linear \PVI
system to the \PV system takes place according to the scheme displayed in 
Fig. \ref{PVI-V.fig}.

\begin{figure}
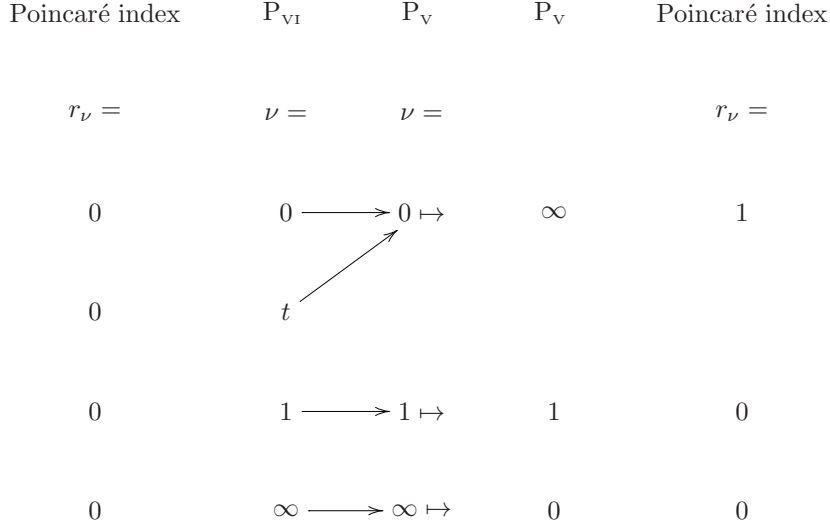

\begin{center}
\begin{xymatrix}{
  \text{Poincar\'e index} & {{\rm P}_{\VI}\;} & {{\rm P}_{\V}\;} & {{\rm P}_{\V}\;} & \text{Poincar\'e index} \\
  {r_{\nu}=} & {\nu=} & {\nu=} & {} & {r_{\nu}=} \\
  0 & 0 \ar[r] & {0\mapsto} & {\infty} & 1 \\
  0 & t \ar[ur] &&  &  \\
  0 & 1 \ar[r] & {1\mapsto} & 1 & 0 \\
  0 & {\infty} \ar[r] & {\infty\mapsto} & 0 & 0 
}
\end{xymatrix}
\vskip0.0cm
\caption{Coalescence of the two regular singularities $ 0,t $ of the \PVI system into 
the irregular singularity at $ \infty $ of a \PV system according to the limit 
transition II of \cite{Ki_1994}.}\label{PVI-V.fig}
\end{center}
\end{figure}

At the level of the \PVI isomonodromic system the following scaling takes place as 
$ n \to\infty $, $ \epsilon \to 0 $
\begin{align}
  \lambda_{\VI} & = \frac{1}{\lambda_{\V}},
  \\
  t_{\VI} & = \epsilon t_{\V},
  \\
  \lim_{\epsilon \to 0} R^{-1}_{t{\VI}}\Psi_{\VI}(t_{\VI},\lambda_{\VI}) & = \Psi_{\V}(t_{\V},\lambda_{\V}),
  \\
  R^{-1}_{t{\VI}}\frac{1}{2}\theta_{\infty {\VI}}\sigma_3 R_{t{\VI}} & = A_{0 {\V}}+O(\epsilon),
  \\
  R^{-1}_{t{\VI}} A_{1 {\VI}} R_{t{\VI}} & = A_{1 {\V}}+O(\epsilon),
  \\
  -R^{-1}_{t{\VI}}\frac{d}{dt_{\V}}R_{t{\VI}} & = 
   \frac{1}{t_{\V}}\left( \frac{1}{2}\theta_{\infty {\V}}\sigma_3+A_{0 {\V}}+A_{1 {\V}} \right)
    + O(\epsilon) .
\label{iReg_coal_IM}
\end{align}

Consequently 
Theorem 2 of Kitaev \cite{Ki_1994} allows us to compute the \PV monodromy data under 
the limit transition II described above.
\begin{theorem}[\cite{Ki_1994}]\label{LimitII}
Let $ \mu_{\VI} \in {\mathcal M}_{\VI}(\theta_{\infty{\V}}-\theta_{6},\theta_{1{\VI}},\theta_{6},\theta_{\infty{\VI}}) $
and suppose the following conditions hold:
\begin{enumerate}
\item
  (\ref{iReg_coal_theta:a}),(\ref{iReg_coal_theta:b}) and 
  $ \theta_{\nu{\VI}} \in \mathbb{C}\backslash\mathbb{Z} $,
\item
  $ (M_{t{\VI}}-e^{-\pi i \theta_{6}})(M_{0{\VI}}-e^{-\pi i (\theta_{\infty{\V}}-\theta_{6})}) \neq  0$,
\item
  $ \alpha, \beta, l \in \mathbb{C}\backslash\mathbb{Z} $, where
\begin{gather}
  T = {\rm Tr}(M_{t{\VI}}-e^{\pi i \theta_{6}})(M_{0{\VI}}-e^{-\pi i (\theta_{\infty{\V}}-\theta_{6})}),
  \\
  l=\frac{1}{\pi i}\log\left(
    \cos\pi\theta_{\infty{\V}}+\frac{1}{2}T\pm\sqrt{(\cos\pi\theta_{\infty{\V}}+\frac{1}{2}T)^2-1}
                       \right),
  \\
  \alpha = -\frac{1}{2}(\theta_{\infty{\V}}-l), \quad
  \beta = -\frac{1}{2}(\theta_{\infty{\V}}+l) ,
\end{gather}
\item
  the inverse monodromy problem for (\ref{PVI_linear}) is solvable 
  for all pairs 
  $ (\mu_{\VI} ,t_{\VI}) $ 
  such that $ t_{\VI}=\epsilon t_{\V}>0 $,
\item
  it is possible to construct the sequence $ A^{2n}_{\nu {\VI}}(\epsilon t_{\V}) $,
  $ n \in\Z^+ $.
\end{enumerate}
Then the data $ \mu_{\V} \in {\mathcal M}_{\V}(\theta_{0{\V}},\theta_{1{\V}},\theta_{\infty{\V}}) $
are given by the formulae
\begin{gather}
   \theta_{0{\V}}=\theta_{\infty{\VI}}, \qquad
   \theta_{1{\V}}=\theta_{1{\VI}} ,
   \\
   \hat{S}_0 = \begin{pmatrix}
         1 & 0 \\ \frac{2\pi i}{\Gamma(1-\alpha)\Gamma(1-\beta)} & 1 
         \end{pmatrix} ,
   \quad
   \hat{S}_1 = \begin{pmatrix}
         1 & -\frac{2\pi i}{\Gamma(\alpha)\Gamma(\beta)}e^{\pi i\theta_{\infty{\V}}} \\ 0 & 1 
         \end{pmatrix} ,
   \\
   \hat{M}_{0{\V}} = KM_{\infty{\VI}}K^{-1}, \qquad
   \hat{M}_{1{\V}} = KM_{1{\VI}}K^{-1} ,
\label{LimitII_M}
\end{gather}
where $ K $ is the unique (up to a sign) solution of the system
\begin{equation}
   KM_{t{\VI}}K^{-1} = \hat{S}_0e^{\pi i\theta_{6}\sigma_3}, \qquad
   KM_{0{\VI}}K^{-1} = e^{-\pi i\theta_{6}\sigma_3}\hat{S}_1e^{\pi i\theta_{\infty{\V}}\sigma_3} .
\label{LimitII_K}
\end{equation}
\end{theorem}

Applying the above Theorem to the spectrum singularity ensemble we have the following result.
\begin{corollary}
In the limit transition of Theorem (\ref{LimitII}) we find the Painlev\'e V parameters 
of the spectrum singularity ensemble are 
\begin{equation}
  \theta_{0{\V}}=\mu+\bar{\omega}, \quad
  \theta_{1{\V}}=-\mu-\omega, \quad 
  \theta_{\infty{\V}}=2\mu-2\omega_1, \quad
  \theta_{6}=-2\omega_1 .
\label{iReg_Mexp}
\end{equation}
The Stokes multipliers are
\begin{equation}
   \hat{s}_0 = \frac{2\pi i}{\Gamma(1-2\omega_1)\Gamma(1+2\mu)}, \quad
   \hat{s}_1 =-\frac{2\pi ie^{\pi i(2\mu-2\omega_1)}}{\Gamma(2\omega_1)\Gamma(-2\mu)} .
\label{iReg_Stokes}
\end{equation}

\vfill\eject
\begin{sideways}
\begin{minipage}[c][12cm][c]{20cm}{
The monodromy matrices are 
\begin{equation}
  \hat{M}_{0{\V}} =
  \begin{pmatrix}
  e^{\pi i(\mu+\bar{\omega})}(1-\xi^*)-e^{\pi i(\mu-\bar{\omega})}2i\sin\pi2\mu
  & 
  -\frac{\displaystyle\Gamma(1+2\mu)}{\displaystyle\Gamma(2\omega_1)}
   e^{i\pi(2\mu+2\omega_1)}\left[ e^{-\pi i(\mu+\bar{\omega})}2i\sin\pi2\mu
                                 +e^{\pi i(-\mu+\bar{\omega})}\xi^* \right]
  \\
  \frac{\displaystyle\Gamma(2\omega_1)}{\displaystyle\Gamma(1+2\mu)}
  e^{\pi i(2\mu-2\omega_1)}
  \left[ 2i\sin\pi(\mu-\bar{\omega})+e^{\pi i(-\mu+\bar{\omega})}\xi^* \right]
  & 
  e^{\pi i(3\mu-\bar{\omega})}+e^{\pi i(\mu+\bar{\omega})}\xi^*
  \end{pmatrix} ,
\label{iReg_M0}
\end{equation}
\begin{equation}
  \hat{M}_{1{\V}} =
  \begin{pmatrix}
  e^{\pi i(\mu+\omega)}+e^{-\pi i(\mu+\omega)}\xi^*
  &
  \frac{\displaystyle\Gamma(1+2\mu)}{\displaystyle\Gamma(2\omega_1)}
  e^{\pi i(-\mu+\bar{\omega})}\xi^*
   \\ 
  -\frac{\displaystyle\Gamma(2\omega_1)}{\displaystyle\Gamma(1+2\mu)}
  e^{-\pi i2\omega_1}\left[ 2i\sin\pi(\mu+\omega)+e^{-\pi i(\mu+\omega)}\xi^* \right]
  & 
  e^{-\pi i(\mu+\omega)}(1-\xi^*)
  \end{pmatrix} ,
\label{iReg_M1}
\end{equation}
\begin{equation}
  \hat{M}_{\infty{\V}} =
  \begin{pmatrix}
  e^{\pi i(2\mu-2\omega_1)}  
  &
  -\frac{\displaystyle 2\pi i}{\displaystyle\Gamma(-2\mu)\Gamma(2\omega_1)}  
   \\ 
  \frac{\displaystyle 2\pi i}{\displaystyle\Gamma(1+2\mu)\Gamma(1-2\omega_1)}
  e^{\pi i(2\mu-2\omega_1)}
  & 
  2\cos\pi(2\mu+2\omega_1)-e^{\pi i(2\mu-2\omega_1)}
  \end{pmatrix} .
\label{iReg_Minfty}
\end{equation}
}
\end{minipage} 
\end{sideways}
\vfill\eject

\end{corollary}
\begin{proof}
The set of \PVI parameters (\ref{SSE_Mexp:A}) in our application imply that the 
limiting \PV parameters are
$ \theta_{6} = -2\omega_1, \theta_{0{\V}}=\mu+\bar{\omega}, \theta_{1{\V}}=-\mu-\omega, 
  \theta_{\infty{\V}}=2\mu-2\omega_1 $, which are consistent with those of 
(\ref{Bulk_Mexp:A}). Using the structure of the \PVI monodromy matrices (\ref{SSE_MA:a})
and (\ref{SSE_MA:b}) we compute that $ \alpha = -\theta_{\infty{\V}}+\theta_{6} $
and $ \beta = -\theta_{6} $, which fixes the Stokes multipliers as given by
(\ref{iReg_Stokes}). We solve the system (\ref{LimitII_K}) which yields the solution
\begin{equation}
  K = K_{1,1}
  \begin{pmatrix}
  1  
  &
  -i\frac{\displaystyle m_{t{\VI}}}{\displaystyle 2\sin\pi\theta_{6}}  
   \\[10pt] 
  \frac{\displaystyle 2i}{\displaystyle \hat{s}_1}
  e^{\pi i(\theta_{\infty{\V}}+\theta_{6})}
  \sin\pi(-\theta_{\infty{\V}}+\theta_{6})
  & 
  \frac{\displaystyle e^{\pi i(\theta_{\infty{\V}}+\theta_{6})}m_{t{\VI}}}
       {\displaystyle \hat{s}_1}
  \end{pmatrix} ,
\end{equation}
for an arbitrary non-zero $ K_{1,1} $. Finally we employ these findings in 
(\ref{LimitII_M}) and after lengthy calculations arrive at explicit forms for the 
monodromy matrices (\ref{iReg_M0},\ref{iReg_M1},\ref{iReg_Minfty}).
\end{proof}

\vfill\eject
\begin{sideways}
\begin{minipage}[c][12cm][c]{20cm}{
However the explicit parameterisation of the alternative \PV monodromy matrices 
(\ref{PV_M0},\ref{PV_M1}) 
allows us to compute these directly from the data deduced from the expansion of the 
$\tau$-function about $ x=0 $, as given by Theorems \ref{Bulk_scaling} and
\ref{Bulk_x=0Mdata}.
\begin{corollary}\cite{Ji_1982,AK_2000}
In terms of the other set of monodromy data 
$ \{M_{0\rm V},M_{1\rm V},M_{\infty\rm V}\} $ we have
\begin{equation}
  M_{0\rm V} = 
  \begin{pmatrix}
    e^{\pi i(-3\mu+\bar{\omega})}(1-\xi^*)
    &
    -\frac{\displaystyle\Gamma(1+2\mu)}{\displaystyle\Gamma(2\omega_1)}
     e^{\pi i(-\mu-\omega)}
    \left[ 2i\sin\pi2\mu+e^{-\pi i2\mu}\xi^* \right]
    \cr
    \frac{\displaystyle\Gamma(2\omega_1)}{\displaystyle\Gamma(1+2\mu)}
     e^{\pi i(-2\mu+2\omega_1)}
    \left[ 2i\sin\pi(\mu-\bar{\omega})+e^{\pi i(-\mu+\bar{\omega})}\xi^* \right]
    &
    2\cos\pi(\mu+\bar{\omega})-e^{\pi i(-3\mu+\bar{\omega})}(1-\xi^*)
  \end{pmatrix} ,
\label{PV_Nm0}
\end{equation}
and
\begin{equation}
  M_{1\rm V} = 
  \begin{pmatrix}
    e^{\pi i(\mu+\omega)}+e^{\pi i(2\omega_1-\mu+\bar{\omega})}\xi^*
    &
    \frac{\displaystyle\Gamma(1+2\mu)}{\displaystyle\Gamma(2\omega_1)}
     e^{\pi i(-\mu+\bar{\omega)}}\xi^*
    \cr
    -\frac{\displaystyle\Gamma(2\omega_1)}{\displaystyle\Gamma(1+2\mu)}
     e^{\pi i2\omega_1}
    \left[ 2i\sin\pi(\mu+\omega)+e^{\pi i(2\omega_1-\mu+\bar{\omega})}\xi^* \right]
    &
    e^{\pi i(-\mu-\omega)}-e^{\pi i(2\omega_1-\mu+\bar{\omega})}\xi^*
  \end{pmatrix} ,
\label{PV_Nm1}
\end{equation} 
and
\begin{equation}
  M_{\infty\rm V} = 
  \begin{pmatrix}
    2\cos\pi(2\mu+2\omega_1)-e^{-\pi i(2\mu-2\omega_1)}
    &
    -\frac{\displaystyle 2\pi i}{\displaystyle\Gamma(-2\mu)\Gamma(2\omega_1)}
    \cr
     \frac{\displaystyle 2\pi i}{\displaystyle\Gamma(1+2\mu)\Gamma(1-2\omega_1)}
     e^{-\pi i(2\mu-2\omega_1)}
    &
    e^{-\pi i(2\mu-2\omega_1)}
  \end{pmatrix} .
\label{PV_NmInf}
\end{equation}
\end{corollary}
\begin{proof}
The first two matrices are computed from (\ref{PV_M0}) and (\ref{PV_M1}) after noting that a
comparison of the Stokes multipliers (\ref{PV_StokesM}) and (\ref{iReg_Stokes})
obliges us to set $ r=-2\mu $. The third matrix is computed using (\ref{PV_MInf}).
As a check one can verify directly that both sets of monodromy matrices
(\ref{iReg_M0},\ref{iReg_M1},\ref{iReg_Minfty}) and (\ref{PV_Nm0},\ref{PV_Nm1},\ref{PV_NmInf}) 
are related by the transformations of (\ref{PV_Mxfm}).
\end{proof}
}
\end{minipage} 
\end{sideways}
\vfill\eject

In the discussion of the asymptotics of the \PV $\sigma$-function as 
$ t \to \pm\infty $ given in \cite{AK_2000} the following parameter 
\begin{equation}
   \beta_0 = \frac{1}{2\pi i}\ln\left\{
             \xi^*[1-e^{\pi i(-2\mu+2\omega_1)}(1-\xi^*)] \right\}
\end{equation}
was found to be crucial. We note that when $ \xi^*=1 $ both $ \beta_0 $ and the 
upper left element of $ M_0 $ both vanish.
A comparison of our monodromy matrices, evaluated at $ \xi^*=1 $, with those
discussed in \cite{AK_1997} where a solution to one of the connection problems 
was reported reveals that our case is precisely the case of the 
{\em lower truncated solution}. In terms of the parameters used in that work we 
have 
\begin{align}
   i\hat{u} &= 2^{2(2\mu-2\omega_1)}e^{-\pi i(\mu-\bar{\omega})}
               \frac{\Gamma(1+2\mu)}{\Gamma(2\omega_1)} , \\
   i\hat{v} &= \sqrt{\frac{2}{\pi}}e^{\pi i(\mu+\omega_1)}\cos\pi(\mu-\bar{\omega}) .
\end{align}
This means that the $ \sigma$-function has the following asymptotic expansion
\begin{equation}
  \zeta(s) \mathop{\sim}\limits_{s \to -i\infty}
  \zeta_0(s)-ie^{\pi i(\mu+\omega_1)}\cos\pi(\mu-\bar{\omega})
              \sqrt{\frac{|s|}{4\pi}}e^{-is/2}
\end{equation}
in the sector $ -\pi \leq {\rm arg}(s) \leq 0 $. 
Here $ \zeta_0(s) $ is the formal, algebraic expansion
\begin{multline}
  \zeta_0(s) \mathop{\sim}\limits_{|s| \to \infty}
  \frac{s^2}{16}+(\mu-\frac{1}{2}i\omega_2)s \\
  + 4\mu^2-2\mu i\omega_2+\omega_1^2+\omega_2^2-\frac{1}{4}
  - 2i\omega_2(4\mu^2-4\omega_1^2)s^{-1} \\
  + \left[ 16\mu^4-8(4(\omega_1^2+\omega_2^2)+1)\mu^2-16\omega_1^2\omega_2^2
           +(4\omega_1^2-1)(4\omega_1^2-4\omega_2^2-1) \right]s^{-2}
  + {\rm O}(s^{-3}) . 
\end{multline}
However this is not the physically interesting asymptotic expansion of 
$ s \to +i\infty $.

\section{Conclusions}
\setcounter{equation}{0}
An especially interesting problem that remains outstanding is the connection 
formula relating the $ x\to 0 $ behaviour of the \PV $\tau$-function 
arising in the bulk scaling of the spectrum singularity ensemble $ \mathcal{A}(x) $, as
given in (\ref{Bulk_tau}), to the $ x\to +i\infty $ behaviour.
The bulk scaling limit of the partition function for the Dyson CUE (\ref{CUE}) is 
the generating function for the gap probability, which has the Fredholm determinant
formula
\begin{equation}
  E((-t,t);\xi) = \det(\mathbb{I}-\xi\mathbb{K}|_{L^2(-t,t)}) ,
\end{equation}
where the integral operator $ \mathbb{K} $ has a kernel with the sine kernel form
\begin{equation}
  K(x,y) = \frac{\sin\pi(x-y)}{\pi(x-y)} .
\end{equation}
This is a special case of our spectral average $ \mathcal{A}(x) $ and the precise 
relation is
\begin{equation}
  E((-t,t);\xi) = \mathcal{A}(4it;\mu = 0,\omega_1 = 0,\omega_2 = 0;\xi) .
\end{equation}
The large $ t\to \infty $ asymptotic expansion of the gap probability 
$ E((-t,t);\xi=1) $ is known to be
\begin{equation}
  E((-t,t);\xi=1) \sim e^{3\zeta'(-1)+\frac{1}{12}\log 2}t^{-1/4}e^{-\frac{1}{2}t^2+{\rm o}(1)} ,
\end{equation}
or alternatively as
\begin{equation}
  t\frac{d}{dt}\log E((-t,t);\xi=1) \sim
  -t^2-\frac{1}{4}-\frac{1}{16t^2}-\frac{5}{32t^4}+\ldots .
\end{equation}
For $ 0 < \xi < 1 $ this result becomes
\begin{equation}
   t\frac{d}{dt}\log E((-t,t);\xi) \sim
   \frac{2t}{\pi}\log(1-\xi) + \frac{\log^2(1-\xi)}{2\pi^2} + {\rm o}(1) .
\end{equation}
This asymptotic problem has been the subject of intense and continuing study
\cite{Dy_1976,Dy_1995,BTW_1992,BTW_1992a,Sh_1995,DIZ_1997,No_2001,Eh_2006,DIKZ_2006}.
The question we pose then is - what is the 
generalisation of this asymptotic result valid for generic $ \omega_1, \omega_2, \mu $
and that is uniformly valid for $ |1-\xi| < \delta $ with some finite $ \delta> 0 $?

Our results do not strictly apply for the situation of moments or singularities that are
negative $ \Re(\mu),\Re(\omega_1) < -1/2 $ and $ |t| \to 1 $, however the present
results may carry over to this case. This case is quite relevant in the 
context of studies concerning the averages of the characteristic polynomial in the 
CUE with negative integer powers\cite{FK_2004}
where such averages are employed to make conjectures
concerning the averages of powers of the zeta or $L$-functions on their critical
lines.

\section*{Acknowledgements}
This work was supported by the Australian Research Council.

\bibliographystyle{plain}
\bibliography{moment,random_matrices,nonlinear}

\end{document}